\documentclass[12pt,a4paper,reqno]{amsart}

\usepackage{amssymb}
\usepackage{amscd}
\usepackage[pdftex,pdfpagelabels]{hyperref}
\usepackage{enumerate}
\usepackage{comment}
\usepackage{graphicx}
\usepackage{siunitx}
\usepackage{tikz-cd}
\usepackage{stix}
\usepackage{bm}
\DeclareMathAlphabet\mathbfcal{LS2}{stixcal}{b}{n}
\numberwithin{equation}{section}

\usepackage{mathtools}
\usepackage[tableposition=top]{caption}
\usepackage{booktabs,dcolumn}

\DeclareFontFamily{OT1}{rsfs}{}
\DeclareFontShape{OT1}{rsfs}{n}{it}{<-> rsfs10}{}
\DeclareMathAlphabet{\mathscr}{OT1}{rsfs}{n}{it}

\theoremstyle{plain}

\newtheorem{theorem}{Theorem}[section]

\newtheorem{lemma}[theorem]{Lemma}

\newtheorem{question}[theorem]{Question}

\theoremstyle{definition}

\newtheorem{remark}[theorem]{Remark}

\newcommand\R{\mathbb{R}}

\newcommand\F{\mathbf{F}}

\newcommand\eps{\varepsilon}

\renewcommand{\mod}{\bmod}

\begin{document}

\title[Point sets with few distinct distances]{Planar point sets with forbidden $4$-point patterns and few distinct distances}

\author{Terence Tao}
\address{UCLA Department of Mathematics, Los Angeles, CA 90095-1555.}
\email{tao@math.ucla.edu}


\subjclass[2020]{52C10, 05B25}

\begin{abstract}  We show that for any large $n$, there exists a set of $n$ points in the plane with $O(n^2/\sqrt{\log n})$ distinct distances, such that any four points in the set determine at least five distinct distances.  This answers (in the negative) a question of Erd\H{o}s.  The proof combines an analysis by Dumitrescu of forbidden four-point patterns with an algebraic construction of Thiele and Dumitrescu (to eliminate parallelograms), as well as a randomized transformation of that construction (to eliminate most other forbidden patterns).
\end{abstract}  

\maketitle


\section{Introduction}

In this note we answer (in the negative) the following question of Erd\H{o}s \cite[p. 101]{erdos-75}, \cite[p. 61]{erdos-83}, \cite[p. 149]{erdos-86}, \cite[p. 34]{erdos-88}, \cite[p. 347]{erdos-95}, \cite[p. 231]{erdos-97}, which is also listed at \cite[Conjecture 6, p. 204]{bmp} and \url{https://www.erdosproblems.com/135}:

\begin{question}[Erd\H{o}s \#135]\label{quest} Let $A \subset \R^2$ be a set of $n$ points such that any four points in the set determine at least five distinct distances. Must $A$ determine $\gg n^2$ many distances?
\end{question}

In fact we show that $A$ can have as few as $O(n^2/\sqrt{\log n})$ distinct distances.  Here we use the usual asymptotic notation $X = O(Y)$, $X \ll Y$, or $Y \gg X$ to denote the bound $|X| \leq CY$ for an absolute constant $C$, and $X \asymp Y$ as shorthand for $X \ll Y \ll X$.  This negative answer to Question \ref{quest} also refutes a stronger claim (see \cite[p. 231]{erdos-97}) that $A$ must contain a subset of cardinality $\gg n$ in which all pairwise distances are distinct.

Our arguments combine probabilistic arguments of Dumitrescu \cite{dumitrescu} with a deterministic algebraic construction of Thiele \cite{thiele} and Dumitrescu \cite{dumitrescu-2008}, each of which had achieved partial progress towards Question \ref{quest}.  The idea of combining these two 
constructions was already anticipated in \cite[Problem 2]{dumitrescu}, but the resulting number-theoretic claims that needed to be verified seem challenging to prove.  Our sole innovation is to randomize the algebraic construction in \cite{thiele}, \cite{dumitrescu-2008} to allow this strategy to become viable without requiring any advanced number theory.

We now establish the negative answer to Question \ref{quest}. In \cite[Lemma 1]{dumitrescu} (see also \cite[Figure 1]{dumitrescu}) it was observed that the four-point configurations $\{P_1,P_2,P_3,P_4\}$ in the plane which failed to determine five or more distinct distances belonged to one of eight patterns:
\begin{itemize}
\item[$\pi_1$:] An equilateral triangle plus an arbitrary vertex.
\item[$\pi_2$:] A parallelogram.
\item[$\pi_3$:] An isosceles trapezoid (four points on a line, $P_1,P_2,P_3,P_4$, where $\overleftrightarrow{P_1P_2} = \overleftrightarrow{P_3P_4}$, form a degenerate isosceles trapezoid).
\item[$\pi_4$:] A star with three edges of the same length.
\item[$\pi_5$:] A path with three edges of the same length.
\item[$\pi_6$:] A kite.
\item[$\pi_7$:] An isosceles triangle plus an edge incident to a base endpoint, and whose length equals the length of the base.
\item[$\pi_8$:] An isosceles triangle plus an edge incident to the apex, and whose length equals the length of the base.
\end{itemize}
We draw the reader's attention to the parallelogram pattern $\pi_2$, as these are far more prevalent than the other seven patterns in many natural constructions, and will thus require particular care to avoid. We permit some overlap between patterns; for instance, a rhombus is simultaneously of pattern $\pi_2$, $\pi_5$, and $\pi_6$.

It is a well-known observation of Erd\H{o}s \cite{erdos-46} that the square grid $ \{0,\dots,n-1\}^2$ determines $\asymp n^2/\sqrt{\log n}$ distinct distances when $n$ is large.  Thus, after adjusting $n$ by a constant if necessary, it will suffice to establish

\begin{theorem}[Main theorem]\label{main-thm} Let $n$ be sufficiently large.  Then there exists a subset of $ \{0,\dots,n-1\}^2$ of cardinality $\gg n$ that avoids all of the eight patterns $\pi_1,\dots,\pi_8$.
\end{theorem}

This theorem answers \cite[Problem 1]{dumitrescu} in the affirmative. (As discussed in Remark \ref{problems} below, it also answers a randomized version of \cite[Problem 2]{dumitrescu} in the affirmative.)

Two ``near-misses'' to Theorem \ref{main-thm} were obtained by Dumitrescu and by Thiele.  In \cite[Theorem 1]{dumitrescu}, Dumitrescu observed that a random subset of $ \{0,\dots,n-1\}^2$ with $\asymp n$ elements would avoid $\pi_1,\pi_3,\dots,\pi_8$ (after deleting some exceptional elements) with positive probability; however, this construction typically contained many parallelograms $\pi_2$.  On the other hand, in earlier works of Thiele \cite{thiele}, and independently Dumitrescu \cite{dumitrescu-2008}, a deterministic algebraic construction of Erd\H{o}s and Tur\'an \cite{erdos-turan} (based on the finite field parabola $\{ (x,x^2): x \in \F_p\}$, and initially used to construct Sidon sets) was used to produce a subset of $\{0,\dots,n-1\}^2$ with $\gg n$ elements which avoided $\pi_1, \pi_2, \pi_3$ (and was furthermore in general position, with no four points concyclic), but potentially contained configurations in the patterns $\pi_4, \dots, \pi_8$; see \cite{dumitrescu} for further discussion.  It was also observed in \cite{dumitrescu} that the entire grid $ \{0,\dots,n-1\}^2$ avoids the pattern $\pi_1$.

As it turns out, the two constructions above can in fact be combined to give Theorem \ref{main-thm}, with the main new idea being to randomize the parabola used in the algebraic construction.  Firstly, by the standard deletion method (also employed in \cite{dumitrescu}), it suffices to establish a weaker version of the main theorem in which a small number of patterns are permitted to survive:

\begin{theorem}[Main theorem with survivors]\label{main-thm-survive} Let $n$ be sufficiently large.  Then there exists a subset $A$ of $ \{0,\dots,n-1\}^2$ of cardinality $\gg n$ that has at most $O(n)$ copies of each of the eight patterns $\pi_1,\dots,\pi_8$.
\end{theorem}

Indeed, if one randomly refines the set $A$ in Theorem \ref{main-thm-survive} by a small constant factor $0 < \eps < 1$, then with positive probability, the refined set will have cardinality $\gg \eps n$ and have at most $O(\eps^4 n)$ surviving configurations in the eight patterns $\pi_1,\dots,\pi_8$. Deleting all points associated with those patterns then gives Theorem \ref{main-thm} after choosing $\eps$ to be a sufficiently small positive constant and adjusting $n$ appropriately (cf. \cite[Chapter 3]{alon}).

Next, we recall a key calculation\footnote{The parameter $n$ in this paper would correspond to $n^2$ in the notation of \cite{dumitrescu}.} in \cite{dumitrescu}, which established that the seven patterns $\pi_1, \pi_3,\dots,\pi_8$ (excluding the parallelogram $\pi_2$) were not too common in the grid $\{0,\dots,n-1\}^2$:

\begin{lemma}\label{lemma-refine}\cite[Lemmas 6, 8, 11]{dumitrescu} There are at most $O(n^5)$ four-point configurations in $ \{0,\dots,n-1\}^2$ that are of one of the patterns $\pi_1, \pi_3,\dots,\pi_8$.
\end{lemma}

In fact, more was shown in \cite{dumitrescu}: the pattern $\pi_1$ does not occur at all in $ \{0,\dots,n-1\}^2$, while $\pi_4, \pi_5, \pi_7, \pi_8$ each appear at most $O(n^{14/3} \log n)$ times.  However, the bound in Lemma \ref{lemma-refine} (which, as shown in \cite{dumitrescu}, is optimal up to constants for $\pi_3$ and $\pi_6$) will be all that is needed for our purposes. 

As observed in \cite{dumitrescu}, the probabilistic method then implies a weaker version of Theorem \ref{main-thm-survive} in which the parallelogram pattern $\pi_2$ is not controlled, since one can simply take a random refinement of the grid $\{0,\dots,n-1\}^2$ by a factor of $1/n$; such a refinement has cardinality $\asymp n$ with high probability, but cuts down the $O(n^5)$ configurations in Lemma \ref{lemma-refine} by a factor of $1/n^4$ on average.  Unfortunately, this argument breaks down quite badly for $\pi_2$, as the number of configurations in $\{0,\dots,n-1\}^2$ of this pattern is instead $\asymp n^6$ (see \cite[Lemma 3]{dumitrescu}), and this random construction usually contains far too many parallelograms to delete completely.
  
In order to eliminate parallelograms, we instead randomize the algebraic construction from \cite{thiele}, \cite{dumitrescu-2008}, which is specifically designed to avoid $\pi_2$.  Namely, we first use Bertrand's postulate to locate a prime $p$ between $4n$ and $8n$ (in particular, $p$ is odd and $p \asymp n$), and let $a,b,c,d,e$ be elements of the finite field $\F_p$ of order $p$, chosen uniformly at random amongst all quintuples that obey the non-degeneracy condition
\begin{equation}\label{adbc}
  ad-bc \neq 0,
 \end{equation}
 thus $(a,b)$, $(c,d)$ are linearly independent vectors in $\F_p^2$.   Note that if we just chose $a,b,c,d,e$ uniformly and independently at random, then the condition \eqref{adbc} would be obeyed with probability $1-O(1/p)$.  So the distribution of $(a,b,c,d,e)$ is the uniform distribution\footnote{Geometrically, the state space $\{(a,b,c,d,e) \in \F_p^5: ad-bc = 0\}$ can be viewed as coordinatizing the group $\mathrm{Aff}(\F_p^2) \equiv \mathrm{GL}_2(\F_p) \ltimes \F_p^2$ of invertible affine transformations $(x,y) \mapsto (ax+by+f, cx+dy+e)$ of the plane $\F_p^2$, quotiented out (on the left) by the Galilean symmetries $(x,y) \mapsto (x+t, y+2tx+t^2)$ of the parabola $y=x^2$ to achieve the normalization $f=0$.  If desired, one could also quotient out by the additional dilation symmetries $(x,y) \mapsto (\lambda x, \lambda^2 x)$ of the parabola for $\lambda \neq 0$ (which in our coordinates becomes $(a,b,c,d,e) \mapsto (\lambda a, \lambda b, \lambda^2 c, \lambda^2 d, \lambda^2 e)$) to projectively remove an additional degree of freedom; but it will be slightly more convenient here to work with affine coordinates instead of projective ones.} on $\F_p^5$ conditioned to an event of probability $1-O(1/p)$.

The random set $A$ we will propose for Theorem \ref{main-thm-survive} will then be defined as the truncated random parabola
\begin{equation}\label{parab}  
  A \coloneqq \{ (x,y) \in \{0,\dots,n-1\}^2: (ax+by)^2 = cx+dy+e \mod p \}. 
\end{equation}
The standard truncated parabola $S \coloneqq \{ (x,y) \in \{0,\dots,n-1\}^2: y = x^2 \mod p\}$ is essentially the example considered in \cite{thiele}, \cite{dumitrescu-2008}, and is the special case of \eqref{parab} when $a=d=1$ and $b=c=e=0$.  Geometrically, $A$ is formed by applying a random invertible affine transformation\footnote{Such a transformation would generate a slightly modified equation $(ax+by+f)^2 = cx+dy+e$ for some randomly chosen $f \in \F_p$, but it is easy to redistribute the effect of that lower order term $f$ to the $c,d,e$ coefficients, thanks to the aforementioned Galilean symmetries of the parabola.} to the parabola $\{ (x,x^2): x \in \F_p \}$ and then restricting to the grid $\{0,\dots,n-1\}^2$.  While the construction \eqref{parab} is not nearly as random as a completely random subset of $\F_p^2$ of density $\asymp 1/p$, we shall see that there will (barely) still be enough ``entropy'' in the five random parameters $a,b,c,d,e$ for the probabilistic method to be effective\footnote{In particular, we will need probability bounds of $O(1/p^4)$ in Lemma \ref{seven-rare} below, so it will be important that the number $p^5 - O(p^4)$ of possible states $(a,b,c,d,e)$ in this random construction is $\gg p^4$, even after quotienting out by dilation symmetry (which cuts down the number of states by a factor of $p$ or so). To put it another way, the parabolae in a plane have four independent degrees of freedom, thus allowing randomization of such parabolae to be an effective tool for controlling four-point patterns.}.  In particular, we avoid the difficult number-theoretic questions of trying to count the number of occurrences of the
patterns $\pi_1,\pi_3,\dots,\pi_8$ in the standard truncated parabola $S$ (cf. \cite[Problem 2]{dumitrescu}).

A routine application of the second moment method reveals that $A$ usually has the ``right'' cardinality (up to accceptable errors):

\begin{lemma}\label{card}  With probability at least $0.9$ (say), the set $A$ has cardinality $n^2/p + O(\sqrt{n})$.  In particular, for $n$ large enough, the cardinality of $A$ is $\asymp n$ with probability at least $0.9$.
\end{lemma}

\begin{proof}  Let us temporarily remove the non-degeneracy condition \eqref{adbc}, so that $a,b,c,d,e$ now become independent random variables.  It is clear that any point $(x,y) \in \{0,\dots,n-1\}^2$ will now lie in $A$ with probability $1/p$, just from the randomness of $e$ alone.  In fact, any two distinct points $(x,y), (x',y') \in \{0,\dots,n-1\}^2$ will both lie in $A$ with a joint probability of $1/p^2$, from the randomness of $c,d,e$ (since $(x,y,1)$ and $(x',y',1)$ are linearly independent in $\F_p$); thus the events $(x,y) \in A$ are pairwise independent.  This implies that the cardinality of $A$ has mean $n^2/p \asymp n$ and variance $O(n^2/p) = O(n)$, and hence from Chebyshev's inequality, $A$ will have cardinality $n^2/p + O(\sqrt{n})$ with probability at least $0.95$ (say).  Conditioning to the event \eqref{adbc}, we obtain the claim. 
\end{proof}

Crucially, the construction also avoids $\pi_2$ (cf. \cite[Lemma 1]{dumitrescu-2008}):

\begin{lemma}\label{avoid} With probability $1$, the set $A$ avoids the parallelogram pattern $\pi_2$.
\end{lemma}

\begin{proof}
Since $p > 4n$, any parallelogram in $\{0,\dots,n-1\}^2$ lifts up to a parallelogram in $\F_p^2$ (defined as a quadruple of distinct points in $\F_p^2$ of the form $\{P, P+H, P+K, P+H+K\}$).  Hence it suffices to show that the finite field parabola
\begin{equation}\label{ffp}
 \{ (x,y) \in \F_p^2 : (ax+by)^2 = cx+dy+e \}
\end{equation}
does not contain parallelograms.  By applying the invertible change of variables $(x,y) \mapsto (cx+dy+e, ax+by)$, which is affine and thus preserves\footnote{In the language of additive combinatorics, this change of variables, as well as the lifting map from $\{0,\dots,n-1\}^2$ to $\F_p^2$, are Freiman isomorphisms.} the property of being parallelogram-free, it suffices to show that the standard parabola
$$ \{ (x,x^2) : x \in \F_p^2 \}$$
does not contain parallelograms.  But this is an easy consequence of the identity
$$ (x+h+k)^2 - (x+h)^2 - (x+k)^2 + x^2 = 2hk$$
for any $x,h,k \in \F_p$, noting that the right-hand side is non-zero whenever $h,k$ are non-zero since $p$ is odd.
\end{proof}

Finally, the construction mostly avoids any given configuration of four points (in a roughly similar fashion to a completely random subset of $\{0,\dots,n-1\}^2$ of density $\asymp 1/p$):

\begin{lemma}\label{seven-rare}  Let $P_1,P_2,P_3,P_4$ be distinct elements of $\{0,\dots,n-1\}^2$.  Then the probability that $\{P_1,P_2,P_3,P_4\}$ lies in $A$ is $O(1/p^4)$.
\end{lemma}

\begin{proof}  By lifting up to $\F_p^2$, it suffices to show that if $\tilde P_1,\tilde P_2,\tilde P_3,\tilde P_4$ are four distinct points in $\F_p^2$, then the probability that they all lie in the random parabola \eqref{ffp} is $O(1/p^4)$.   If three of these points are collinear, then the probability is in fact zero, since by Bezout's theorem, a line can only\footnote{It is important here that the parabola is non-degenerate, which is a consequence of the condition \eqref{adbc}.  If the random set $A$ was permitted to also be a straight line (or a pair of straight lines) by weakening or removing this condition, then this lemma would break down if $\tilde P_1, \tilde P_2, \tilde P_3, \tilde P_4$ were collinear.  This explains why it was convenient to impose \eqref{adbc} at the beginning of the construction.} intersect a parabola in at most two points.  Thus we may assume that $\tilde P_1,\tilde P_2,\tilde P_3,\tilde P_4$ are in general position (no three collinear).
  
  Note that the distribution of the random parabola \eqref{ffp} is invariant with respect to invertible affine transformations of $\F_p^2$.  Thus, we may normalize $\tilde P_1$ to be the origin $(0,0)$, then $\tilde P_2$ to be $(1,0)$, then (by general position) $\tilde P_3$ to be $(0,1)$.  By general position again, the fourth point $\tilde P_4$ can then be written as $(s,t)$ where $s,t$ are non-zero\footnote{The general position hypothesis also implies that $s+t \neq 1$, but we will not need this additional constraint here.} elements of $\F_p$.  Substituting into \eqref{ffp}, our task is to show that the system of equations
\begin{align*}
  0 &= e \\
  a^2 &= c + e \\
  b^2 &= d + e \\
  (as+bt)^2 &= cs + dt + e
\end{align*}
are simultaneously satisfied by $(a,b,c,d,e)$ with probability $O(1/p^4)$. We can eliminate some variables and express the above system in the reduced form
\begin{align*}
  (s^2-s) a^2 + 2st ab + (t^2-t) b^2 &= 0 \\ 
  c &= a^2 \\
  d &= b^2 \\
  e &= 0.
\end{align*}
Since $s,t$ are non-zero and $p$ is odd, the homogeneous quadratic polynomial $(a,b) \mapsto (s^2-s) a^2 + 2st ab + (t^2-t) b^2 $ does not vanish identically, and thus has $O(p)$ roots in $\F_p^2$, as can be seen either by direct calculation (the quadratic formula) or an appeal to the Schwarz--Zippel lemma \cite{dl}, \cite{schwartz}, \cite{zippel}.  The other three equations determine $c,d,e$ in terms of $a,b$, hence there are $O(p)$ solutions $(a,b,c,d,e) \in \F_p^5$ to the above system.  Since $(a,b,c,d,e)$ has the uniform distribution in $\F_p^5$, conditioned to an event \eqref{adbc} of probability $1-O(1/p)$, we obtain the desired bound of $O(1/p^4)$.
\end{proof}

From Lemma \ref{seven-rare}, Lemma \ref{lemma-refine}, and linearity of expectation, the expected number of configurations in one of the seven patterns $\pi_1,\pi_3,\dots,\pi_8$ in $A$ is $O(p^5/p^4) = O(n)$, hence by Markov's inequality this number will be $O(n)$ with probability at least $0.9$ (say).  Theorem \ref{main-thm-survive} (and hence Theorem \ref{main-thm}) then follows with probability at least $0.8$ from this observation, Lemma \ref{avoid}, and Lemma \ref{card}.

\begin{remark} As in \cite{thiele}, \cite{dumitrescu-2008}, the sets constructed here for Theorem \ref{main-thm} (or Theorem \ref{main-thm-survive}) are in general position (no three points collinear), and also have no four points concyclic, basically because any non-degenerate parabola in $\F_p^2$ also has these properties (over $\F_p$).
\end{remark}

\begin{remark}\label{problems}  Two problems from \cite{dumitrescu} remain open.  The first (see \cite[Problem 2]{dumitrescu}) is to show that the standard truncated parabola $S$ contains a subset of cardinality $\gg n$ that avoids all eight patterns $\pi_1,\dots,\pi_8$; we show this for a randomized variant $A$ of this parabola, but the original problem will likely require more advanced techniques from analytic number theory to resolve. The second (see \cite[Problem 3]{dumitrescu} and \cite[p. 35]{erdos-88}) is to show that the quantity $\phi(n,4,5)$, which one can define as the minimal number of distinct distances determined by an $n$-point set in which any $4$ points determine at least $5$ distances, grows super-linearly in $n$.  The general lower bound of Guth and Katz \cite{guth-katz} on the Erd\H{o}s distinct distances problem \cite{erdos-46} gives $\phi(n,4,5) \gg n / \log n$, but actually the trivial bound of $\phi(n,4,5) \geq \frac{n-1}{2}$, which comes from observing that each distance to a fixed (or ``pinned'') reference point $P_0$ in the set can occur at most twice if the star pattern $\pi_4$ is forbidden, is asymptotically superior in this case.  The construction here establishes the upper bound $\phi(n,4,5) \ll n^2 / \sqrt{\log n}$, but makes no new progress on the lower bound.  See also \cite[Table 1]{dumitrescu} for some other upper and lower bounds on $\phi(n,k,l)$ for various $k,l$.
\end{remark}

\subsection{Acknowledgements}  

The author is supported by NSF grant DMS-2347850.


\begin{thebibliography}{10}

\bibitem{alon}
N. Alon, J. Spencer, The Probabilistic Method, Fourth ed., Wiley, New York, 2016.

\bibitem{bmp}
P. Bra{\ss}, W. Moser, and J. Pach, Research Problems in Discrete Geometry, Springer, New York, 2005.

\bibitem{dl}
R. DeMillo, R. Lipton, \emph{A probabilistic remark on algebraic program testing}, Information Processing Letters \textbf{7} (4) (1978) 193--195.

\bibitem{dumitrescu-2008}
A. Dumitrescu, \emph{On distinct distances among points in general position and other related problems}, Period. Math. Hungar. \textbf{57} (2) (2008) 165--176.

\bibitem{dumitrescu}
A. Dumitrescu, \emph{Distinct distances in planar point sets with forbidden 4-point patterns}, Discrete Mathematics, \textbf{343} (9) (2020), 111967.

\bibitem{erdos-46}
P. Erd\H{o}s, \emph{On sets of distances of n points}, Amer. Math. Monthly \textbf{53} (1946) 248--250.

\bibitem{erdos-75}
P. Erd\H{o}s, \emph{On some problems of elementary and combinatorial geometry}, Annali di Matematica Pura ed Applicata IV, Vol. CIII (1975), 99--108.

\bibitem{erdos-83}
P. Erd\H{o}s, \emph{Extremal problems in number theory, combinatorics and geometry}, in: Proceedings of the International Congress of Mathematicians August 16--24, 1983, Warszawa, pp. 51--70.

\bibitem{erdos-86}
P. Erd\H{o}s, \emph{On some metric and combinatorial geometric problems}, Discrete Math. \textbf{60} (1986), 147–153.

\bibitem{erdos-88}
P. Erd\H{o}s, \emph{Some old and new problems in combinatorial geometry}, in: Applications of Discrete Mathematics (Clemson, SC, 1986), SIAM, Philadelphia, PA, 1988, pp. 32--37.

\bibitem{erdos-95}
P. Erd\H{o}s, \emph{Some of my recent problems in combinatorial number theory, geometry and combinatorics}, in: Y. Alavi, et al. (Eds.), Graph Theory, Combinatorics, Algorithms and Applications, Vol. 1, Wiley, 1995, pp. 335--349.

\bibitem{erdos-97}
P. Erd\H{o}s, \emph{Some old and new problems in various branches of combinatorics}, Discrete Math. \textbf{165/166} (1997), 227--231.

\bibitem{erdos-turan}
P. Erd\H{o}s, P. Tur\'an, \emph{On a problem of Sidon in additive number theory, and on some related problems}, J. London Math. Soc. \textbf{16} (1941), 212--215.

\bibitem{guth-katz}
L. Guth, N. Katz, \emph{On the Erdős distinct distances problem in the plane}, Ann. of Math. \textbf{181} (2015) 155--190.

\bibitem{schwartz}
J. T. Schwartz, \emph{Fast probabilistic algorithms for verification of polynomial identities}, Journal of the ACM \textbf{27} (4) (1980) 701--717.

\bibitem{thiele}
T. Thiele, \emph{The no-four-on-circle problem}, J. Combin. Theory Ser. A \textbf{71} (1995) 332--334.

\bibitem{zippel}
R. Zippel, \emph{Probabilistic algorithms for sparse polynomials}, in: Symbolic and Algebraic Computation, EUROSAM '79, An International Symposium on Symbolic and Algebraic Computation, Marseille, France, June 1979, Proceedings, Lecture Notes in Computer Science, Vol. 72, Springer, 1979, pp. 216--226.


\end{thebibliography}
\end{document}